\documentclass[11pt]{article}
\usepackage{amsmath,amssymb,amsthm}
\usepackage{mathtools}
\usepackage{mathdots}
\usepackage{hyperref,url}
\usepackage{float}
\usepackage{tikz}
\usepackage{subcaption}

\hypersetup{
	colorlinks=true,
  linkcolor=black,          
  citecolor=black,         
  filecolor=black,      
  urlcolor=black           
}

\newcommand{\seqnum}[1]{\href{https://oeis.org/#1}{\rm \underline{#1}}}

\DeclareMathOperator{\black}{black}
\DeclareMathOperator{\white}{white}

\allowdisplaybreaks

\title{Dyck paths on black-and-white lattices}
\author{ Sela Fried\\ Department of Computer Science, Israel Academic College\\ Ramat Gan, Israel\\ \texttt{friedsela@gmail.com} }

\date{}

\theoremstyle{plain}
\newtheorem{theorem}{Theorem}
\newtheorem{lemma}{Lemma}
\newtheorem{corollary}{Corollary}

\begin{document}
\maketitle

\begin{abstract}
A Dyck path of semilength $n$ is a lattice path from $(0,0)$ to $(n,n)$ consisting of $n$ right-steps $(1,0)$ and $n$ up-steps $(0,1)$ that never rises above the line $y=x$. These paths are enumerated by the Catalan numbers and play a central role in enumerative combinatorics. We color the cells of the integer grid in black and white according to two natural patterns, namely chessboard and column-alternating, and enumerate the Dyck paths having equal numbers of black and white cells beneath them.
\end{abstract}

\vspace{10px}


\vspace{10px}

\section{Introduction}

Let $n$ be a positive integer. A Dyck path of semilength $n$ is a lattice path from $(0,0)$ to $(n,n)$ consisting of $n$ right-steps $(1,0)$ and $n$ up-steps $(0,1)$ that never rises above the line $y=x$ (see Figure \ref{f11}). Dyck paths are classical objects in enumerative combinatorics, since they are counted by the Catalan numbers 
\[C_n=\frac{1}{n+1}\binom{2n}{n}.\]

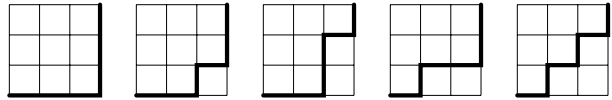
\begin{figure}[H]
\centering
\scalebox{0.4}{%
\begin{tikzpicture}
\tikzset{DyckPath/.style={black,line width=4pt,line cap=round}}
\newcommand{\GridThree}{%
  \draw[very thin] (0,0) grid (3,3);}
\begin{scope}[shift={(0.0,0.0)}]    
  \GridThree
  \draw[DyckPath] (0,0)--(1,0)--(2,0)--(3,0)--(3,1)--(3,2)--(3,3);
\end{scope}
\begin{scope}[shift={(4.2,0.0)}]    
  \GridThree
  \draw[DyckPath] (0,0)--(1,0)--(2,0)--(2,1)--(3,1)--(3,2)--(3,3);
\end{scope}
\begin{scope}[shift={(8.4,0.0)}]     
  \GridThree
  \draw[DyckPath] (0,0)--(1,0)--(2,0)--(2,1)--(2,2)--(3,2)--(3,3);
\end{scope}
\begin{scope}[shift={(12.6,0.0)}]    
  \GridThree
  \draw[DyckPath] (0,0)--(1,0)--(1,1)--(2,1)--(3,1)--(3,2)--(3,3);
\end{scope}
\begin{scope}[shift={(16.8,0.0)}]    
  \GridThree
  \draw[DyckPath] (0,0)--(1,0)--(1,1)--(2,1)--(2,2)--(3,2)--(3,3);
\end{scope}
\end{tikzpicture}}
\caption{The five Dyck paths of semilength \(3\).}\label{f11}
\end{figure}

The purpose of this work is to study Dyck paths through the lens of colored lattices. Specifically, we consider two natural ways of coloring the cells of the integer lattice (see Figure \ref{ff22}):

\begin{enumerate}
\item [ (a)] Chessboard coloring - the $(i,j)$th cell is black if $i+j$ is even and white otherwise.
\item [(b)] Column-alternating coloring - the $j$th column is black if $j$ is odd and white otherwise.
\end{enumerate}
\begin{figure}[H]
\centering
\scalebox{0.4}{%
\begin{tikzpicture}
\newcommand{\ChessFour}{
  \foreach \i in {0,...,3}{
    \foreach \j in {0,...,3}{\pgfmathtruncatemacro\p{mod(\i+\j,2)}
      \ifnum\p=0 \fill[black!20] (\i,\j) rectangle ++(1,1);\fi}}
  \draw[very thin] (0,0) grid (4,4);
}
\newcommand{\ColsFour}{
  \foreach \i in {0,...,3}{
    \ifnum\i=0 \foreach \j in {0,...,3}{\fill[black!20] (\i,\j) rectangle ++(1,1);} \fi
    \ifnum\i=2 \foreach \j in {0,...,3}{\fill[black!20] (\i,\j) rectangle ++(1,1);} \fi}
  \draw[very thin] (0,0) grid (4,4);
}
\begin{scope}[shift={(0,0)}]
  \ChessFour
  \node[font=\sffamily\Huge] at (2,-1.4) {(a) Chessboard coloring.};
\end{scope}
\begin{scope}[shift={(12.0,0)}] 
  \ColsFour
  \node[font=\sffamily\Huge] at (2,-1.4) {(b) Column-alternating coloring.};
\end{scope}
\end{tikzpicture}}
\caption{The two colorings studied in this work.}\label{ff22}
\end{figure}

Given a Dyck path $p$, let $\black(p)$ and $\white(p)$ denote the numbers of black and white cells that lie below $p$. We say that $p$ is black-white balanced (with respect to a given coloring, which should be clear from context) if $\black(p)=\white(p)$. Our goal is to determine, for each of the two colorings, how many Dyck paths of semilength $n$ are black-white balanced. For example, exactly $6$ of the $14$ Dyck paths of semilength $4$ are black-white balanced with respect to the chessboard coloring and only $3$ are black-white balanced with respect to the column-alternating coloring (see Figures \ref{fig2} and \ref{tt94}).

When dealing with objects enumerated by the Catalan numbers, refinements according to specific statistics give rise to the Narayana numbers. These are defined by \[
N(n,k)=\frac{1}{n}\binom{n}{k}\binom{n}{k-1},\quad 1\leq k\leq n,
\] and satisfy 
\[\sum_{k=1}^nN(n,k)=C_n.\] 

For example, among the multitude of objects enumerated by the Catalan numbers (e.g., \cite{Sta}), a fundamental one is the family of words, called Dyck words, containing $n$ pairs of parentheses, which are correctly matched. Refining and counting only Dyck words having exactly $k$ subwords $()$, we obtain the Narayana number $N(n,k)$.

Thus, the problem of enumerating black-white balanced Dyck paths is a special case of a combinatorial statistic. Among the few works related we are aware of \cite{Ch07} by Cheng et al., who studied the area of Catalan paths on checkerboard (Catalan paths are very similar to Dyck paths but not the same). More recently, Fried \cite{Fr25} studied the distribution of black and white cells in bargraphs of $k$-ary words and permutations. 

We prove that the black-white balanced condition in the chessboard coloring corresponds to a well-known statistic on Dyck paths, namely the number of up-steps occurring at odd positions. As a consequence, the number of black-white balanced Dyck paths of semilength $n$ is given by the Narayana number $N(n,\lceil n/2 \rceil)$.  

In the column-alternating coloring the black-white balanced condition forces the right-steps to appear in consecutive pairs (with one extra right-step at the beginning when $n$ is odd). This structural constraint allows us to transform the problem into the terminology of generalized Dyck paths, 
leading to closed formulas in terms of binomial coefficients.

Notation wise, for a condition $c$, we denote by $\mathbf{1}_c$ the indicator function returning $1$ if $c$ holds and $0$, otherwise. Furthermore, if $m$ is a positive integer, we denote by $[m]$ the set $\{1,2,\ldots,m\}$.

\section{Chessboard coloring}

\begin{theorem}\label{c12}
The number of black-white balanced Dyck paths of semilength $n$ is $N(n,\lceil n/2 \rceil)$.
In particular, the proportion of black-white balanced Dyck paths among all Dyck paths of semilength $n$
is asymptotically $2/\sqrt{\pi n}$.
\end{theorem}

\begin{proof}
For $i\in[n]$ denote by $p_i$ the number of cells in the $i$th column below $p$.
Let $\black_i(p)$ and $\white_i(p)$ denote the numbers of black and white cells that lie in the $i$th column below $p$.
We claim that
\begin{equation}\label{ee1}
\black_i(p)-\white_i(p)=(-1)^{i-1}\mathbf{1}_{\textnormal{$p_i$ is odd}}.
\end{equation}
Indeed, if $p_i$ is even, then the numbers of black and white cells are equal and therefore their difference is $0$.
If $p_i$ is odd, then there is one more black cell than white cells if $i$ is odd and one more white cell than black cells if $i$ is even.
Summing \eqref{ee1} over $i\in[n]$ gives
\begin{equation}\label{e2}
\black(p)-\white(p)=\sum_{i=1}^{n}(-1)^{i-1}\mathbf{1}_{\textnormal{$p_i$ is odd}}.
\end{equation}

Now let $i\in[n]$ and let $t_i\in[2n]$ be the index of the $i$-th right-step.
Before the $i$-th right-step there are $i-1$ right-steps and $p_i$ up-steps.
Thus $t_i=(i-1)+p_i+1=i+p_i$, and therefore, by \eqref{e2},
\begin{align*}
\black(p)-\white(p)
&=\sum_{i=1}^{n}(-1)^{i-1}\mathbf{1}_{\textnormal{$p_i$ is odd}}\\
&=\sum_{i=1}^{n}(-1)^{i-1}\mathbf{1}_{\textnormal{$t_i-i$ is odd}}\\
&=\sum_{\substack{i\in[n]\\ \textnormal{$i$ is odd}}}\mathbf{1}_{\textnormal{$t_i$ is even}}
-\sum_{\substack{i\in[n]\\ \textnormal{$i$ is even}}}\mathbf{1}_{\textnormal{$t_i$ is odd}}\\
&=\sum_{\substack{i\in[n]\\ \textnormal{$i$ is odd}}}\mathbf{1}_{\textnormal{$t_i$ is even}}
-\sum_{\substack{i\in[n]\\ \textnormal{$i$ is even}}}\bigl(1-\mathbf{1}_{\textnormal{$t_i$ is even}}\bigr)\\
&=\sum_{i=1}^n\mathbf{1}_{\textnormal{$t_i$ is even}}-\sum_{\substack{i\in[n]\\ \textnormal{$i$ is even}}}1\\
&=|\{\textnormal{right-steps of $p$ at even positions}\}|-\lfloor n/2\rfloor.
\end{align*}

Following the standard convention in \cite{Sul05}, index the steps of $p$ by $0,1,\dots,2n-1$.
Let $k$ be the number of up-steps at odd indices.
Then the number of right-steps at even positions in the $1,2,\dots,2n$ indexing is exactly $n-k$,
so the displayed identity becomes
\[
\black(p)-\white(p)=(n-k)-\lfloor n/2\rfloor=\lceil n/2\rceil-k.
\]
Hence $p$ is black-white balanced if and only if $k=\lceil n/2\rceil$.

By \cite[Proposition~2]{Sul05}, the number of Dyck paths of semilength $n$ with $k$ up-steps at odd positions
is the Narayana number $N(n,k)$. Therefore the number of black-white balanced Dyck paths is
$N(n,\lceil n/2\rceil)$.

Finally, the stated asymptotic for $N(n,\lceil n/2\rceil)/C_n$ follows by a standard application of Stirling's approximation.
\end{proof}

\begin{corollary}
The number of black-white balanced Dyck paths of semilength $n$ is $N(n,\lceil n/2\rceil)$.
\end{corollary}

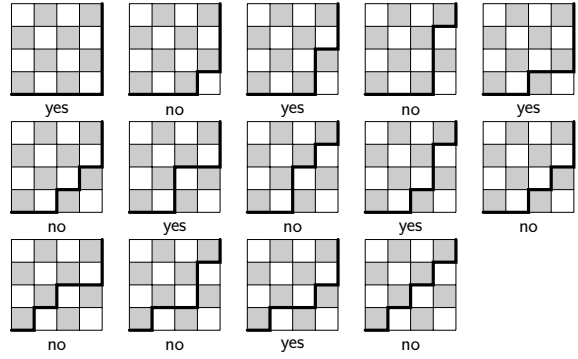
\begin{figure}[H]
\centering
\scalebox{0.3}{%
\begin{tikzpicture}
\tikzset{DyckPath/.style={black,line width=4pt,line cap=round}}
\newcommand{\ChessFour}{%
  \foreach \i in {0,...,3}{
    \foreach \j in {0,...,3}{\pgfmathtruncatemacro\p{mod(\i+\j,2)}
      \ifnum\p=0 \fill[black!20] (\i,\j) rectangle ++(1,1);\fi}}
  \draw[very thin] (0,0) grid (4,4);}
\newcommand{\Yes}{\node[font=\sffamily\Huge] at (2,-0.7) {yes};}
\newcommand{\No}{\node[font=\sffamily\Huge]  at (2,-0.7) {no};}

\begin{scope}[shift={(0.0,10.4)}]
  \ChessFour
  \draw[DyckPath] (0,0)--(1,0)--(2,0)--(3,0)--(4,0)--(4,1)--(4,2)--(4,3)--(4,4);
  \Yes
\end{scope}
\begin{scope}[shift={(5.2,10.4)}]
  \ChessFour
  \draw[DyckPath] (0,0)--(1,0)--(2,0)--(3,0)--(3,1)--(4,1)--(4,2)--(4,3)--(4,4);
  \No
\end{scope}
\begin{scope}[shift={(10.4,10.4)}]
  \ChessFour
  \draw[DyckPath] (0,0)--(1,0)--(2,0)--(3,0)--(3,1)--(3,2)--(4,2)--(4,3)--(4,4);
  \Yes
\end{scope}
\begin{scope}[shift={(15.6,10.4)}]
  \ChessFour
  \draw[DyckPath] (0,0)--(1,0)--(2,0)--(3,0)--(3,1)--(3,2)--(3,3)--(4,3)--(4,4);
  \No
\end{scope}
\begin{scope}[shift={(20.8,10.4)}]
  \ChessFour
  \draw[DyckPath] (0,0)--(1,0)--(2,0)--(2,1)--(3,1)--(4,1)--(4,2)--(4,3)--(4,4);
  \Yes
\end{scope}

\begin{scope}[shift={(0.0,5.2)}]
  \ChessFour
  \draw[DyckPath] (0,0)--(1,0)--(2,0)--(2,1)--(3,1)--(3,2)--(4,2)--(4,3)--(4,4);
  \No
\end{scope}
\begin{scope}[shift={(5.2,5.2)}]
  \ChessFour
  \draw[DyckPath] (0,0)--(1,0)--(2,0)--(2,1)--(2,2)--(3,2)--(4,2)--(4,3)--(4,4);
  \Yes
\end{scope}
\begin{scope}[shift={(10.4,5.2)}]
  \ChessFour
  \draw[DyckPath] (0,0)--(1,0)--(2,0)--(2,1)--(2,2)--(3,2)--(3,3)--(4,3)--(4,4);
  \No
\end{scope}
\begin{scope}[shift={(15.6,5.2)}]
  \ChessFour
  \draw[DyckPath] (0,0)--(1,0)--(2,0)--(2,1)--(3,1)--(3,2)--(3,3)--(4,3)--(4,4);
  \Yes
\end{scope}
\begin{scope}[shift={(20.8,5.2)}]
  \ChessFour
  \draw[DyckPath] (0,0)--(1,0)--(2,0)--(2,1)--(3,1)--(3,2)--(4,2)--(4,3)--(4,4);
  \No
\end{scope}

\begin{scope}[shift={(0.0,0.0)}]
  \ChessFour
  \draw[DyckPath] (0,0)--(1,0)--(1,1)--(2,1)--(2,2)--(3,2)--(4,2)--(4,3)--(4,4);
  \No
\end{scope}
\begin{scope}[shift={(5.2,0.0)}]
  \ChessFour
  \draw[DyckPath] (0,0)--(1,0)--(1,1)--(2,1)--(3,1)--(3,2)--(3,3)--(4,3)--(4,4);
  \No
\end{scope}
\begin{scope}[shift={(10.4,0.0)}]
  \ChessFour
  \draw[DyckPath] (0,0)--(1,0)--(1,1)--(2,1)--(3,1)--(3,2)--(4,2)--(4,3)--(4,4);
  \Yes
\end{scope}
\begin{scope}[shift={(15.6,0.0)}]
  \ChessFour
  \draw[DyckPath] (0,0)--(1,0)--(1,1)--(2,1)--(2,2)--(3,2)--(3,3)--(4,3)--(4,4);
  \No
\end{scope}

\end{tikzpicture}}
\caption{The $14$ Dyck paths of semilength $4$ with chessboard coloring. The label under each Dyck path indicates whether it is black-white balanced. Exactly $6=N(4,2)$ Dyck paths are labeled with `yes', in agreement with Theorem \ref{c12}.}
\label{fig2}
\end{figure}

\section{Column-alternating coloring}

\begin{lemma}\label{tt09}
Let $p$ be a Dyck path of semilength $n$ and for $i\in[n-1]$ let $u_i$ be the number of up-steps between the $i$th right-step and the $i+1$th right-step. Then $p$ is black-white balanced if and only if 
\begin{equation}
\begin{cases}
u_1=u_3=\cdots=u_{n-1}=0,& \textnormal{if $n$ is even},\\
u_2=u_4=\cdots=u_{n-1}=0,& \textnormal{if $n$ is odd}.
\end{cases}
\end{equation} 
\end{lemma}

\begin{proof}
Let $p$ be a Dyck path of semilength $n$ and for $i\in[n]$ denote by $p_i$ the number of cells in the $i$th column below $p$. Clearly
\begin{equation}\label{e8}
\black(p)-\white(p)=\sum_{i=1}^n(-1)^{i-1}p_i.
\end{equation}
Now, $u_i=p_{i+1}-p_i$. Since $p_1=0$, telescoping gives  $p_j=\sum_{i=1}^{j-1}u_i$, for each $j\in[n]$.
Substituting this into \eqref{e8} gives
\begin{align*}
\black(p)-\white(p)&=\sum_{j=1}^n(-1)^{j-1}\sum_{i=1}^{j-1}u_i\\
&=\sum_{i=1}^{n-1}u_i\sum_{j=i+1}^{n}(-1)^{j-1}\\
&=\sum_{i=1}^{n-1}u_{i}\frac{(-1)^{i}-(-1)^{n}}{2}\\
&=\sum_{i=1}^{n-1}u_i\mathbf{1}_{\textnormal{$n-i$ is odd}}.
\end{align*} Thus, $p$ is black-white balanced if and only if \[\sum_{i=1}^{n-1}u_i\mathbf{1}_{\textnormal{$n-i$ is odd}}=0.\]
This condition is clearly equivalent to the statement of the lemma.
\end{proof}

\begin{theorem}\label{tt94}
The number of black-white balanced Dyck paths of semilength $n$ is given by
\[
\begin{cases}
\frac{1}{2m+1}\binom{3m}{m},&\textnormal{if $n=2m$},\\
\frac{1}{m+1}\binom{3m+1}{m},&\textnormal{if $n=2m+1$}.
\end{cases}
\] In particular, the proportion of black-white balanced Dyck paths among all Dyck paths of semilength $n$ is asymptotically $\beta\alpha^n$, where $\alpha = 3\sqrt{3}/8$ and $\beta=\sqrt{3/2}$, if $n$ is even and $\beta=\sqrt{2}$, if $n$ is odd.
\end{theorem}

\begin{proof}
Assume that $n=2m$.
By Lemma \ref{tt09} we have $u_1=u_3=\cdots=u_{n-1}=0$. Thus, the right-steps come in consecutive pairs (see Figure \ref{fig5} (a)).
Collapse each pair of right-steps into a single double right-step. We obtain a path from $(0,0)$ to $(2m,2m)$ consisting of $m$ double right-steps $(2,0)$ and $2m$ up-steps $(0,1)$, that never rises above the line $y=x$. Rescale the $x$-axis by a factor of $2$ to obtain a path from $(0,0)$ to $(m,2m)$ with steps $(1,0)$ and $(0,1)$ that never rises above the line $y=2x$. It is well known (e.g., \seqnum{A001764} in \cite{Sl}) that the number of such paths is given by
$\frac{1}{2m+1}\binom{3m}{m}$.

Assume now that $n=2m+1$.
By Lemma \ref{tt09} we have $u_2=u_4=\cdots=u_{n-1}=0$. Thus, except for the first right-step, all right-steps come in consecutive pairs (see Figure \ref{fig5} (b)).
Thus, after the first right-step has been made, we may collapse the remaining pairs of right-steps into double right-steps. This gives us a path that may be regarded as a path from $(0,0)$ to $(2m, 2m+1)$ consisting of $m$ double right-steps $(2,0)$ and $2m+1$ up-steps $(0,1)$ that never rises above the line $y=x+1$. Rescaling the $x$-axis by a factor of $2$ yields a path from $(0,0)$ to $(m,2m+1)$ with steps $(1,0)$ and $(0,1)$ that never rises above the line $y=2x+1$. We claim that the number of such paths is equal to the number of paths from $(0,0)$ to $(2m+1,m)$ with steps $(1,0)$ and $(0,1)$ that never rise above the line $x=2y$. Indeed, reversing the path and interchanging $(1,0)$ and $(0,1)$ is a bijection between the two path families (see Figure \ref{fifi55}). It is well known (e.g., \seqnum{A006013} in \cite{Sl}) that the number of such paths is given by
$\frac{1}{m+1}\binom{3m+1}{m}$.
\end{proof}

\begin{figure}[H]
\centering
\renewcommand{\thesubfigure}{\alph{subfigure}}
\newcommand{\UpStack}[2]{%
  \draw[->] (#1,0.00) -- ++(0,0.90);
  \node at (#1,1.10) {$\vdots$};
  \draw[->] (#1,1.30) -- ++(0,0.90);
  \node[above] at (#1,2.22) {#2};
}
\subfloat[Even $n$.]{
\scalebox{0.7}{%
\begin{tikzpicture}[thick]
\draw[->] (0.00,0) -- (0.90,0);   
\draw[->] (1.06,0) -- (1.96,0);   
\draw[->] (2.12,0) -- (3.02,0);   
\draw[->] (3.18,0) -- (4.08,0);   
\draw[->] (4.24,0) -- (5.14,0);   
\draw[->] (5.30,0) -- (6.20,0);   
\UpStack{2.04}{$u_2$}  %
\UpStack{4.16}{$u_4$}  %
\node at (6.90,0) {$\cdots$};
\draw[->] (7.60,0) -- (8.50,0);   
\draw[->] (8.66,0) -- (9.56,0);   
\UpStack{9.64}{$u_n$}
\end{tikzpicture}}%
\label{fig:u-even}}
\\[0.6em]  
\subfloat[Odd $n$.]{
\scalebox{0.7}{%
\begin{tikzpicture}[thick]
\draw[->] (0.00,0) -- (0.90,0);   
\draw[->] (1.06,0) -- (1.96,0);   
\draw[->] (2.12,0) -- (3.02,0);   
\draw[->] (3.18,0) -- (4.08,0);   
\draw[->] (4.24,0) -- (5.14,0);   
\UpStack{0.98}{$u_1$}  %
\UpStack{3.10}{$u_3$}  %
\node at (5.84,0) {$\cdots$};
\draw[->] (6.54,0) -- (7.44,0);   
\draw[->] (7.60,0) -- (8.50,0);  
\UpStack{8.58}{$u_n$}  %
\end{tikzpicture}}%
\label{fig:u-odd}}
\caption{The distribution of the right-steps in black-white balanced Dyck paths, for even and odd $n$.}
\label{fig5}
\end{figure}
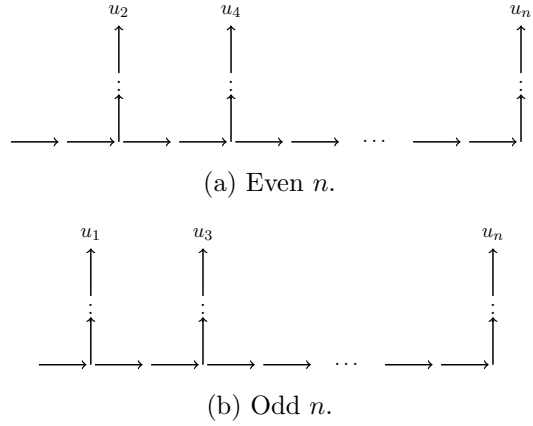

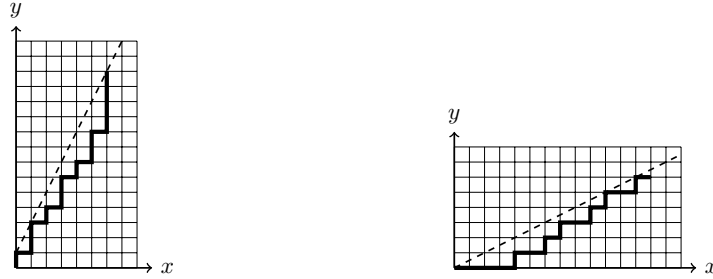
\begin{figure}[H]
    \centering
    \begin{subfigure}[t]{0.5\textwidth}
        \centering
        \captionsetup{width=5.9cm}
        \scalebox{0.8}{%
\begin{tikzpicture}[scale=0.25, thick]
  \useasboundingbox (0,0) rectangle (8,15);
  \draw[very thin] (0,0) grid (8,15);
  \draw[->] (0,0) -- (0,16) node[above] {$y$};
  \draw[->] (0,0) -- (9,0) node[right] {$x$};
  \draw[dashed,domain=0:7] plot (\x,{2*\x+1});
  \draw[line width=2pt]
    (0,0)--(0,1)--(1,1)--(1,2)--(1,3)--(2,3)
    --(2,4)--(3,4)--(3,5)--(3,6)--(4,6)
    --(4,7)--(5,7)--(5,8)--(5,9)--(6,9)
    --(6,10)--(6,11)--(6,12)--(6,13);
\end{tikzpicture}}%
        \caption{A path from $(0,0)$ to $(6,13)$ with steps $(1,0)$ and $(0,1)$ that never rises above the line $y=2x+1$.}
    \end{subfigure}%
    ~ 
    \begin{subfigure}[t]{0.5\textwidth}
        \centering
        \captionsetup{width=5.5cm}
        \scalebox{0.8}{%
\begin{tikzpicture}[scale=0.25, thick]
  \useasboundingbox (0,0) rectangle (15,8);
  \draw[very thin] (0,0) grid (15,8);
  \draw[->] (0,0) -- (0,9) node[above] {$y$};
  \draw[->] (0,0) -- (16,0) node[right] {$x$};
  \draw[dashed,domain=0:7.5] plot ({2*\x},\x);
  \draw[line width=2pt]
    (0,0)--(1,0)--(2,0)--(3,0)--(4,0)--(4,1)
    --(5,1)--(6,1)--(6,2)--(7,2)--(7,3)
    --(8,3)--(9,3)--(9,4)--(10,4)--(10,5)
    --(11,5)--(12,5)--(12,6)--(13,6);
\end{tikzpicture}}%
        \caption{The corresponding path from $(0,0)$ to $(13,6)$ with steps $(1,0)$ and $(0,1)$ that never rises above the line $x=2y$.}
    \end{subfigure}
    \caption{An illustration of the bijection described in the proof of Theorem \ref{tt94} in the odd $n$ case.}\label{fifi55}
\end{figure}

\begin{figure}[H]
\centering
\scalebox{0.3}{%
\begin{tikzpicture}
\tikzset{DyckPath/.style={black,line width=4pt,line cap=round}}
\newcommand{\ColsFour}{%
  \foreach \i in {0,...,3}{
    \ifnum\i=0 \foreach \j in {0,...,3}{\fill[black!20] (\i,\j) rectangle ++(1,1);} \fi
    \ifnum\i=2 \foreach \j in {0,...,3}{\fill[black!20] (\i,\j) rectangle ++(1,1);} \fi}
  \draw[very thin] (0,0) grid (4,4);}
\newcommand{\Yes}{\node[font=\sffamily\Huge] at (2,-0.7) {yes};}
\newcommand{\No}{\node[font=\sffamily\Huge]  at (2,-0.7) {no};}

\begin{scope}[shift={(0.0,10.4)}]
  \ColsFour
  \draw[DyckPath] (0,0)--(1,0)--(2,0)--(3,0)--(4,0)--(4,1)--(4,2)--(4,3)--(4,4);
  \Yes
\end{scope}
\begin{scope}[shift={(5.2,10.4)}]
  \ColsFour
  \draw[DyckPath] (0,0)--(1,0)--(2,0)--(3,0)--(3,1)--(4,1)--(4,2)--(4,3)--(4,4);
  \No
\end{scope}
\begin{scope}[shift={(10.4,10.4)}]
  \ColsFour
  \draw[DyckPath] (0,0)--(1,0)--(2,0)--(3,0)--(3,1)--(3,2)--(4,2)--(4,3)--(4,4);
  \No
\end{scope}
\begin{scope}[shift={(15.6,10.4)}]
  \ColsFour
  \draw[DyckPath] (0,0)--(1,0)--(2,0)--(3,0)--(3,1)--(3,2)--(3,3)--(4,3)--(4,4);
  \No
\end{scope}
\begin{scope}[shift={(20.8,10.4)}]
  \ColsFour
  \draw[DyckPath] (0,0)--(1,0)--(2,0)--(2,1)--(3,1)--(4,1)--(4,2)--(4,3)--(4,4);
  \Yes
\end{scope}

\begin{scope}[shift={(0.0,5.2)}]
  \ColsFour
  \draw[DyckPath] (0,0)--(1,0)--(2,0)--(2,1)--(3,1)--(3,2)--(4,2)--(4,3)--(4,4);
  \No
\end{scope}
\begin{scope}[shift={(5.2,5.2)}]
  \ColsFour
  \draw[DyckPath] (0,0)--(1,0)--(2,0)--(2,1)--(2,2)--(3,2)--(4,2)--(4,3)--(4,4);
  \Yes
\end{scope}
\begin{scope}[shift={(10.4,5.2)}]
  \ColsFour
  \draw[DyckPath] (0,0)--(1,0)--(2,0)--(2,1)--(2,2)--(3,2)--(3,3)--(4,3)--(4,4);
  \No
\end{scope}
\begin{scope}[shift={(15.6,5.2)}]
  \ColsFour
  \draw[DyckPath] (0,0)--(1,0)--(2,0)--(2,1)--(3,1)--(3,2)--(3,3)--(4,3)--(4,4);
  \No
\end{scope}
\begin{scope}[shift={(20.8,5.2)}]
  \ColsFour
  \draw[DyckPath] (0,0)--(1,0)--(2,0)--(2,1)--(3,1)--(3,2)--(4,2)--(4,3)--(4,4);
  \No
\end{scope}

\begin{scope}[shift={(0.0,0.0)}]
  \ColsFour
  \draw[DyckPath] (0,0)--(1,0)--(1,1)--(2,1)--(2,2)--(3,2)--(4,2)--(4,3)--(4,4);
  \No
\end{scope}
\begin{scope}[shift={(5.2,0.0)}]
  \ColsFour
  \draw[DyckPath] (0,0)--(1,0)--(1,1)--(2,1)--(3,1)--(3,2)--(3,3)--(4,3)--(4,4);
  \No
\end{scope}
\begin{scope}[shift={(10.4,0.0)}]
  \ColsFour
  \draw[DyckPath] (0,0)--(1,0)--(1,1)--(2,1)--(3,1)--(3,2)--(4,2)--(4,3)--(4,4);
  \No
\end{scope}
\begin{scope}[shift={(15.6,0.0)}]
  \ColsFour
  \draw[DyckPath] (0,0)--(1,0)--(1,1)--(2,1)--(2,2)--(3,2)--(3,3)--(4,3)--(4,4);
  \No
\end{scope}

\end{tikzpicture}}
\caption{The $14$ Dyck paths of semilength $4$ with column-alternating coloring. The label under each Dyck path indicates whether it is black-white balanced. Exactly $3=\frac{1}{2\cdot2+1}\binom{3\cdot2}{2}$ Dyck paths are labeled with `yes', in agreement with Theorem \ref{tt94}.}
\end{figure}
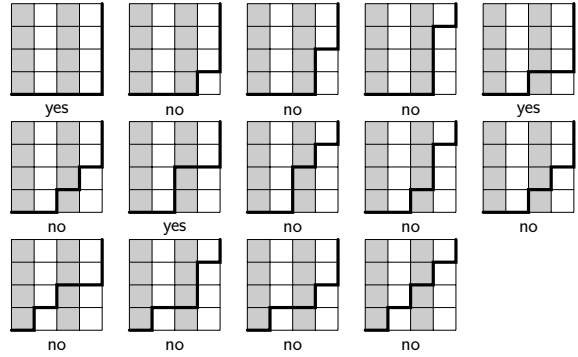

\end{document}